\pdfoutput=1  

\documentclass{amsart}
\usepackage{amssymb}

\usepackage{graphicx}

\newtheorem{theorem}{Theorem}[section]
\newtheorem{lemma}[theorem]{Lemma}

\newtheorem{proposition}[theorem]{Proposition}
\newtheorem{claim}[theorem]{Claim}

\newtheorem{corollary}[theorem]{Corollary}
\newtheorem{question}[theorem]{Question}

\theoremstyle{definition}

\theoremstyle{remark}
\newtheorem{remark}[theorem]{Remark}

\numberwithin{equation}{section}

\begin{document}

\title{Twisted right-angled Artin groups embedded  in knot groups}


\author[K. Himeno]{Keisuke Himeno}
\address{Graduate School of Advanced Science and Engineering, Hiroshima University,
1-3-1 Kagamiyama, Higashi-hiroshima, 7398526, Japan}
\email{himeno-keisuke@hiroshima-u.ac.jp}
\thanks{The first author was supported by JST SPRING, Grant Number JPMJSP2132. }
\author[M. Teragaito]{Masakazu Teragaito}
\address{Department of Mathematics and Mathematics Education, Hiroshima University,
1-1-1 Kagamiyama, Higashi-hiroshima 7398524, Japan.}
\email{teragai@hiroshima-u.ac.jp}
\thanks{The second author has been partially supported by JSPS KAKENHI Grant Number JP20K03587.}


\date{\today}



\begin{abstract}
Twisted right-angled Artin groups are defined through presentation based on mixed graphs.
Each vertex corresponds to a generator,  each undirected edge yields a commuting relation and
each directed edge gives a Klein bottle relation.
If there is no directed edge, then this reduces to an ordinary right-angled Artin group.

There is a characterization of right-angled Artin groups which can be embedded in knot groups by Katayama.
In this paper, 
we completely determine twisted right-angled Artin groups embedded in knot groups.
\end{abstract}


\maketitle


\section{Introduction}\label{sec:intro}

Twisted right-angled Artin groups (abbreviated as TRAAGs) are introduced in \cite{CE} (also \cite{P}) as a natural generalization of right-angled Artin groups (RAAGs).
For TRAAGs, there are a few of recent results \cite{ABP, F}.
In particular, there is a characterization of TRAAGs with torsion, or left ordering.

Recall the definition of TRAAG (see \cite{ABP, F}).
Let $\Gamma$ be a mixed graph $(V,E,D)$, where $\bar{\Gamma}=(V,E)$ is a (finite) simple graph, and
$D$ is a subset of $E$.
An element of $D$ is called a directed edge.
If a directed edge has a tail $x$ and a head $y$, then we denote it by $[x,y\rangle$.
An undirected edge connecting $x$ and $y$ is written as $[x,y]$.
Then the \textit{twisted right-angled Artin group based on $\Gamma$} is
defined as
\[
A(\Gamma)=\langle V \mid xy=yx\ \text{if $[x,y]\in E-D$, and}\ xyx=y\ \text{if $[x,y\rangle\in D$} \rangle.
\]
If $D=\varnothing$, then $A(\Gamma)$ is the ordinary right-angled Artin group based on $\Gamma$.
The second type of  relation $xyx=y$ is called the \textit{Klein  relation} (\cite{F}).

 Droms \cite{D} gave a complete characterization of RAAGs which are $3$-manifold groups, that is, the fundamental groups
 of connected $3$-manifolds.
 Precisely, a RAAG $A(\Gamma)$ is a $3$-manifold group if and only if each connected component of $\Gamma$ 
is a tree or a triangle $K_3$.
This leads us to a natural question.

\begin{question}\label{qu:base}
Which TRAAGs are $3$-manifold groups?
\end{question}

As the simplest example, let $\Gamma$ be a single  arc.
More precisely, it consists of two vertices and a single directed edge $[a,b\rangle$.
Then $A(\Gamma)=\langle a,b \mid aba=b\rangle$ is isomorphic to the fundamental group of the Klein bottle.
(The Klein relation originates from this.)
Hence $A(\Gamma)$ can be realized as the fundamental group of the twisted $I$-bundle over the Klein bottle.
In general, the answer to Question \ref{qu:base} is widely open.

Katayama \cite{K1} gives a complete characterization of RAAGs which embed into a knot group, the fundamental group
of the complement of a knot in the $3$-sphere.
For readers' convenience, we state his result below (Theorem \ref{thm:katayama}).

In general, if a group $G$ has a subgroup isomorphic to a group $H$, then
we say that $H$ embeds into $G$.
If there is no confusion, then we often write $H\le G$.

Let $K$ be a knot in the $3$-sphere $S^3$.
The knot exterior $E(K)$ admits the torus decomposition, or the JSJ decomposition (see \cite{JS,J,Pu}), where
each piece is hyperbolic or Seifert fibered.
We remark that a Seifert fibered piece is either a composing space, a torus knot exterior or a cable space.
Any of these admits the unique Seifert fibration.
If there are two Seifert fibered pieces glued along a common boundary torus, then
the pair is called a \textit{Seifert-Seifert gluing}.

A complete bipartite graph $K_{1,n}$ with $n\ge 1$ is called a star $T_n$.
A path with $n$ vertices is $P_n$. 
In particular, $P_1$ is a single vertex.
A forest is a graph containing no cycles.
Hence every connected component of a forest is a tree.
These are undirected graphs.
The disjoint union of two graphs $\Gamma_1$ and $\Gamma_2$ is denoted by $\Gamma_1+\Gamma_2$.
Also, $m\Gamma$ denotes  the disjoint union of $m$ copies of $\Gamma$.

Thus $A(\Gamma_1+\Gamma_2)=A(\Gamma_1)*A(\Gamma_2)$.
We note that $A(mP_1)=F_m$,  a free group of rank $m$, $A(P_2)=\mathbb{Z}^2$,
and $A(T_n)=F_n\times \mathbb{Z}$.

\begin{theorem}[\cite{K1}]\label{thm:katayama}
Let $K$ be a non-trivial knot in $S^3$ with exterior $E(K)$ and knot group $G(K)$.
Let $\Gamma$ be an undirected graph, and $A(\Gamma)$ the associated RAAG.
\begin{itemize}
\item[(1)]
If $E(K)$ consists of only hyperbolic pieces, then $A(\Gamma)$ embeds into $G(K)$
if and only if $\Gamma=mP_1+nP_2$ for $m,n \ge 0$.
\item[(2)]
If $E(K)$ is Seifert fibered, that is, $K$ is a torus knot, then
$A(\Gamma)$ embeds into $G(K)$ if and only if
$\Gamma$ is either $mP_1$ or a single star $T_n$ for $m,n\ge 1$.
\item[(3)]
If $E(K)$ contains both of a Seifert fibered piece and a hyperbolic piece, and there is no Seifert-Seifert gluing, then $A(\Gamma)$ embeds into $G(K)$ if and only if 
$\Gamma=mP_1+\sum  T_{n_i}$ with $m\ge 0$.  (Possibly, there is no star.)
\item[(4)]
If $E(K)$ contains a Seifert-Seifert gluing, then $A(\Gamma)$ embeds into $G(K)$ if and only if 
$\Gamma$ is a forest.
\end{itemize}
\end{theorem}

For the unknot, the knot group is an infinite cyclic group.
Hence only $A(P_1)=\mathbb{Z}$ can embed there.

In this paper, we focus on TRAAGs embedded in a knot group, and give a complete characterization of such TRAAGs.
To state our main theorem, we add some terminologies.

For an integer $n\ge 1$, a \textit{sink star\/} $S_n$ is a star digraph with $n+1$ vertices and $n$ directed edges which share the same head.
See Figure \ref{fig:star}.

\begin{figure}[htpb]
\includegraphics*[bb=0 0 132 103, scale=0.6]{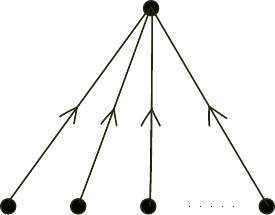}
\caption{A sink star $S_n$ with $n+1$ vertices and $n$ directed edges.}\label{fig:star}
\end{figure}

If $S_n$ has the central vertex $a$ (of degree $n$) and end vertices $b_1,b_2,\dots, b_n$, then
the associated TRAAG $A(S_n)$ has a presentation
\[
\langle a,b_1, b_2, \dots, b_n \mid b_iab_i=a \ (i=1,2,\dots, n) \rangle.
\]

A (non-trivial) torus knot of type $(2p,q)$ is said to be \textit{of even type}, and
its exterior is referred to as a \textit{torus knot exterior of even type}.
Similarly, the exterior of a torus knot of type $(2p,q)$, which lies on the boundary of a smaller solid torus $S^1\times D_0$ in a solid torus $S^1\times D^2$ with $D_0\subset D$
and runs $2p$ times along $S^1$, is called a \textit{cable space of even type}.
(In this case, $q=\pm 1$ is possible.)
A \textit{Seifert fibered piece of even type\/} means either of them.

\begin{theorem}\label{thm:main}
Let $\Gamma$ be a mixed graph with at least one directed edge, and let  $A(\Gamma)$ be the TRAAG based on $\Gamma$.
Let $K$ be a non-trivial knot in $S^3$ with exterior $E(K)$ and knot group $G(K)$.
\begin{itemize}
\item[(1)]
If $E(K)$ consists of only hyperbolic pieces, then $G(K)$ cannot contain $A(\Gamma)$ as a subgroup.
\item[(2)]
If $E(K)$ is Seifert fibered, that is, $K$ is a torus knot, then 
$A(\Gamma)$ embeds into $G(K)$ if and only if $K$ is a (non-trivial) torus knot of even type and
 $\Gamma$ is a single sink star $S_n$ with $n\ge 1$.
\item[(3)]
If $E(K)$ contains both of a Seifert fibered piece and a hyperbolic piece, and there is no Seifert-Seifert gluing, then $A(\Gamma)$ embeds into $G(K)$ if and only if
there is at least one Seifert fibered piece of even type, and $\Gamma=mP_1+\sum T_{m_i}+\sum S_{n_j}$
for $m\ge 0$, $m_i,n_j\ge 1$ and $\Gamma$ has at least one sink star.
\item[(4)]
If $E(K)$ contains a Seifert-Seifert gluing, then  
$A(\Gamma)$ embeds into $G(K)$ if and only if
there is at least one Seifert fibered piece of even type, and $\Gamma=\sum S_{n_i}+F$ for $n_i\ge 1$, where $F$ is a forest and
$\Gamma$ has at least one sink star. (Possibly, $F$ is empty.)
\end{itemize}
\end{theorem}

We remark that there is no upper bound for the number of stars or sink stars in cases (3) and (4) above.
Since we assume that $\Gamma$ has at least one directed edge,
no TRAAG can embed into the knot group of the unknot (see Section \ref{sec:pre}).

Throughout the paper, we use the following notation.
In a group $G$, the commutator between $a$ and $b$ is $[a,b]=a^{-1}b^{-1}ab$ (although the symbol is the same as an undirected edge,
it may not cause a confusion), and
the conjugate of $a$ with $b$ is $a^b=b^{-1}ab$.
For a subgroup $H\le G$ and $g\in G$,  $H^g$ denotes the  conjugate subgroup $g^{-1}Hg$.

If an element $g\in G$ satisfies the relation $g^c=g^{-1}$ for some $c$, and $g\ne 1$,
then $g$ is called a \textit{generalized torsion element of order two\/} \cite{HMT2} (or a reversible element \cite{DD}).
If both of $g$ and $c$ belong to a subgroup $H$ of $G$, then
the pair $(g,c)$ is called a \textit{generalized torsion pair in $H$} \cite{HMT2}.

\section{Preliminaries}\label{sec:pre}

We put the following assumption on mixed graphs throughout the paper.

\medskip\noindent
\textbf{Assumption.} Any mixed graph $\Gamma=(V,E,D)$ has at least one directed edge.
That is, $D\ne \varnothing$.
\medskip

\begin{lemma}[\cite{ABP}]\label{lem:induce}
Let $U\subset V$, and let $\Gamma_U$ be
the induced subgraph spanned by $U$ in $\Gamma$.
Then $A(\Gamma_U)$ is isomorphic to the subgroup $\langle U \rangle$ generated by $U$ in $A(\Gamma)$.
\end{lemma}

The corresponding claim for RAAGs is also well known (see \cite{K1}).

The next lemma is crucial to our argument, and it will be used repeatedly.

\begin{lemma}[\cite{HMT2}]\label{lem:kb}
Let $\Gamma$ be a mixed graph consisting of a single directed edge $[a,b\rangle$.
If $A(\Gamma)$ embeds into a knot group $G(K)$, 
then $E(K)$ contains a Seifert fibered piece $M$ of even type such that
$A(\Gamma)$ is conjugate into $\pi_1(M)$.
More precisely, $A(\Gamma)^x\le \pi_1(M)$ for some $x\in G(K)$, and
$b^x=e^{p}$, where $e$ is the (unique) exceptional fiber of index $2p$ in $M$.
\end{lemma}

Under the assumption of Lemma \ref{lem:kb},
$A(\Gamma)$ is the fundamental group of the Klein bottle.
Theorem 1.4 of \cite{HMT2} shows that $b^x\in \pi_1(M)$ for some Seifert fibered piece $M$ of $E(K)$,
and Theorem 1.10 of \cite{HMT2} (and its proof) describes $b^x=e^p$.
Thus $b^{2x}=e^{2p}$ gives a regular fiber of $M$, which is central in $\pi_1(M)$.

By Lemmas \ref{lem:induce} and \ref{lem:kb}, we see that the knot group of the unknot does not admit a TRAAG.
Also, any TRAAG itself cannot be a knot group, since the abelianization of a TRAAG (with a directed edge)
contains a $2$-torsion \cite{F}.

 Theorem \ref{thm:main} (1) immediately follows from these lemmas.

\begin{corollary}\label{cor:h}
If the knot exterior $E(K)$ consists of only hyperbolic pieces,
then $G(K)$ does not admit a TRAAG as a subgroup.

In particular, the knot group of a hyperbolic knot cannot contain a TRAAG.
\end{corollary}

\begin{proof}
Suppose that $A(\Gamma) \le G(K)$ for a mixed graph $\Gamma$.
By our assumption, $\Gamma$ has a directed edge, which gives an induced subgraph.
Thus $E(K)$ contains a Seifert fibered piece by Lemmas \ref{lem:induce} and \ref{lem:kb}.
\end{proof}

For a mixed graph $\Gamma$, the \textit{underlying graph\/} $\bar{\Gamma}$ is just the graph
obtained from $\Gamma$ by forgetting the orientation of all directed edges.

First, we want to exclude triangles in a mixed graph when the corresponding TRAAG embeds into a knot group.
There are 7 possibilities of triangles as shown in Figure \ref{fig:triangle}.

\begin{figure}[htpb]
\includegraphics*[bb=0 0 361 196, scale=0.8]{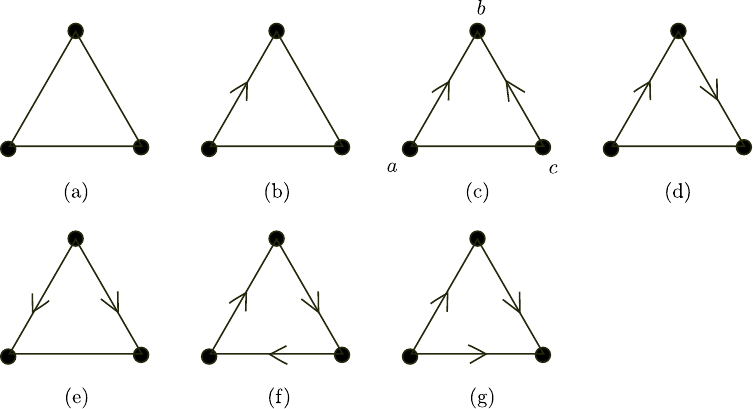}
\caption{Triangles in a mixed graph.}
\label{fig:triangle}
\end{figure}

\begin{lemma}\label{lem:triangle}
Let $\Gamma=(V,E,D)$  be a mixed graph.
Assume that $A(\Gamma)$ embeds into a knot group.
Then the underlying graph $\bar{\Gamma}$ of $\Gamma$ cannot contain a triangle.
\end{lemma}

\begin{proof}
Assume that the underlying graph $\bar{\Gamma}$ contains a triangle, and
let $\Delta$ be the corresponding subgraph of $\Gamma$.
Let $\{a,b,c\}$ be the vertices.
Note that $\Delta$ is the induced subgraph of $\Gamma$.
Hence $A(\Delta)\le A(\Gamma)$ by Lemma \ref{lem:induce}.
We eliminate all possible types as shown in Figure \ref{fig:triangle}.

(a) corresponds to $\mathbb{Z}^3$, which is impossible by \cite[Theorem 5.4.2]{N}.
(b) corresponds to $N\times \mathbb{Z}$, where $N$ is the fundamental group of the Klein bottle.
Since $\mathbb{Z}^2\le N$, we have $\mathbb{Z}^3$ in a knot group, impossible again.


Consider a triangle $\Delta$ of type (c).
Let $[a,b\rangle$, $[c,b\rangle$ and $[c,a]$ be the edges.
Then $\langle a,c\rangle =\mathbb{Z}^2$, and $b^2$ commutes with $a$ and $c$.
We claim $\langle b^2 \rangle \cap \langle a,c\rangle=\{1\}$.
Suppose not.
Then $b^{2m}=a^i c^j$ for some integers $m\ne 0, i,j $.
Recall that we have relations $aba=b$ and $cbc=b$.
This means that $a^b=a^{-1}$ and $c^b=c^{-1}$.
By taking a conjugate of $b^{2m}=a^i c^j$ with $b$, we have
$b^{2m}=a^{-i}c^{-j}$.
Thus $a^{2i}c^{2j}=1$, which implies $i=j=0$.
So, $b^{2m}=1$, which contradicts the fact that any knot group is torsion-free.
We have thus shown that $\langle a,c,b^2\rangle =\mathbb{Z}^3$.
This is impossible as before.

For (d), let $[a,b\rangle$, $[b,c\rangle$ and $[c,a]$ be the edges.
Then $A(\Delta)=\langle a, b,c \mid aba=b, bcb=c, ca=ac\rangle$.
Consider the centralizer $C(a)$ of $a$ in the knot group $G(K)$.
It contains $\langle a, c\rangle =\mathbb{Z}^2$.
Since $ab^2=b^2a$, $b^2$ also belongs to $C(a)$.

We claim that $b^2\not \in \langle a, c\rangle$.
Assume $b^2 \in  \langle a, c\rangle$.
Then $b^2=a^ic^j$ for some integers $i, j$.
By conjugating with $c$, $(b^2)^c=a^i c^j$.
On the other hand, $b^c=b^{-1}$ implies $(b^2)^c=b^{-2}$.
Hence $b^{-2}=a^i c^j$, so $b^4=1$, a contradiction.

Thus we have shown that $C(a)$ is bigger than $\mathbb{Z}^2$.
This implies that there exists a Seifert fibered piece $M$ of the knot exterior $E(K)$ (with respect to the torus decomposition)
and $x\in G(K)$ such that 
$a^x\in \pi_1(M)\le G(K)$ and 
$x^{-1}C(a)x$ is equal to the centralizer of $a^x$ in $\pi_1(M)$ (\cite{AFW,JS}).

For simplicity of notation, we keep using the same symbols after taking conjugations with $x$.
Thus, $C(a)\le \pi_1(M)$, and $C(a)$ is bigger than $\mathbb{Z}^2$.
Hence $a$ is a power of a regular fiber of $M$ (see the paragraph after Theorem 2.5.2 of \cite{AFW}), and
$C(a)=\pi_1(M)$.
There are only three possibilities of $M$; a torus knot exterior, a cable space or a composing space.

Note that $\pi_1(M)$ contains the Klein bottle group $\langle b^2,c \mid b^2cb^2=c\rangle$.
($A(\Delta)$ has a subgroup $\langle b,c \mid bcb=c\rangle$, which contains
a subgroup $\langle b^2, c \mid b^2cb^2=c\rangle$.
By taking a  conjugate with $x$, we find the above subgroup of $\pi_1(M)$.)
This means that $\pi_1(M)$ contains a generalized torsion element of order two.
Hence $M$ is not a composing space \cite{HMT2}, and has an (unique) exceptional fiber $e$ of even index $2p$.

However, \cite[Theorem 1.10]{HMT2} claims that $c$ is conjugate to a power of the exceptional fiber $e$ in $M$.
In fact, $c^y=e^p$ for some $y\in \pi_1(M)$.
Recall that $a$ is a power of a regular fiber $h$.  So, $a=h^i$ for some integer $i\ne 0$, and $a$ is central in $\pi_1(M)$.
Since $e$ has index $2p$, we have $h=e^{2p}$.
Thus
\[
c^{2i}=(e^{2pi})^{y^{-1}}=(h^i)^{y^{-1}} = (h^{y^{-1}})^i=h^i=a.
\]
This contradicts that $\langle a,c\rangle=\mathbb{Z}^2$ in $\pi_1(M)$.

Consider (e).
Similarly, let $[b,a\rangle$, $[b,c\rangle$ and $[c,a]$ be the edges.
Since $\langle a, c\rangle=\mathbb{Z}^2$, 
it has a subgroup $\langle a^2,c^2 \rangle =\mathbb{Z}^2$.

We claim $\langle b\rangle \cap \langle a^2,c^2 \rangle =\{1\}$.
Suppose not.  Then $b^{m}=a^{2i} c^{2j}$ for some integers $m\ne 0$, $i$ and $j$.
Recall $b^a=b^{-1}$.
By taking a conjugate with $a$, we have $b^{-m}=a^{2i}c^{2j}$, so $b^{2m}=1$.
This is impossible.
Since $b$ commutes with $a^2$ and $c^2$, $\langle a^2,c^2,b\rangle =\mathbb{Z}^3$, a contradiction.

(f) leads to a torsion by \cite{ABP}, impossible.

Finally, consider (g).
There are three edges $[a,b\rangle$, $[b,c\rangle$ and $[a,c\rangle$.
Then $\langle a,c \rangle$ is the Klein bottle group, which has an index two subgroup $\langle a,c^2\rangle =\mathbb{Z}^2$.
Consider the centralizer $C(c^2)$.  It contains $\langle a,c^2\rangle =\mathbb{Z}^2$ and $b$.

We claim $b\not \in \langle a, c^2 \rangle$.
Suppose not.
Then $b=a^i c^{2j}$ for some integers $i$ and $j$.
Taking a conjugate with $c$ gives $b^{-1}=a^{-i} c^{2j}$.
Hence $c^{4j}=1$, which gives $j=0$.
Thus $b=a^i$.
However, this contradicts that $\langle a, b\rangle$ is the Klein bottle group.

Thus $C(c^2)$ is bigger than $\mathbb{Z}^2$.
As in (d),  there exists a Seifert fibered piece $M$ of $E(K)$, and we can assume that $C(c^2)\le \pi_1(M)$, after conjugations.
Then, $c^2$ is a power of a regular fiber, and $C(c^2)=\pi_1(M)$.

On the other hand, $C(c^2)$ contains the Klein bottle group $\langle a, b \rangle$.
Hence $\pi_1(M)$ has a generalized torsion element  of order two.
As in (d), $b$ is conjugate to a power of an exceptional fiber in $M$.
This leads to a contradiction that $\langle b, c^2\rangle=\mathbb{Z}^2$ in $\pi_1(M)$ as in (d).
\end{proof}


\begin{lemma}\label{lem:pride}
For a mixed graph $\Gamma$, assume that $A(\Gamma)$ embeds into a knot group.
Then 
the RAAG  $A(\bar{\Gamma})$ based on the underlying graph $\bar{\Gamma}$  is a subgroup of $A(\Gamma)$.
\end{lemma}

\begin{proof}
This follows from \cite{P}.
Let $x_1,x_2,\dots, x_s$ be the generators of $A(\Gamma)$.
Then $A(\Gamma)$ satisfies the conditions $C(4)$ and $T(4)$ of \cite{P}, by Lemma \ref{lem:triangle}.
Theorem 2(iv) of \cite{P}  claims that the squares $x_1^2,x_2^2,\dots, x_s^2$ generate
a subgroup of $A(\Gamma)$ whose the only defining relations are $[x_i^2,x_j^2]=1$ if $[x_i,x_j]=1$.
Clearly, this subgroup is isomorphic to $A(\bar{\Gamma})$.
\end{proof}

As remarked before, Droms \cite{D} shows that
a RAAG $A(\Gamma)$ is a $3$-manifold group if and only if each connected component of $\Gamma$ is
a tree or $K_3$.
Since we consider only knot groups, we can say more as in \cite{K1}.

\begin{lemma}[\cite{K1}]\label{lem:droms}
If a RAAG $A(\Gamma)$ embeds into a knot group, then each connected component of $\Gamma$ is a tree.
Hence $\Gamma$ is a forest.
\end{lemma}

\begin{proof}
Note $A(K_3)=\mathbb{Z}^3$.
Since $\mathbb{Z}^3$ cannot embed into a knot group (see \cite[Theorem 5.4.2]{N}),
the conclusion immediately follows from the above claim of Droms \cite{D}.
\end{proof}

We can prove an analogous result for TRAAGs owing to Lemma \ref{lem:pride}.

\begin{lemma}\label{lem:forest}
If a TRAAG $A(\Gamma)$ embeds into a knot group $G(K)$, then the underlying graph $\bar{\Gamma}$  of $\Gamma$ is a forest.
\end{lemma}

\begin{proof}
Assume that $\bar{\Gamma}$ contains a cycle  for a contradiction.
We can assume that the vertices of the cycle span an induced cycle $C$.
Then $A(\Gamma)$ has a subgroup $A(C)$.
By Lemma \ref{lem:pride}, $A(\bar{C})\le A(C)\le A(\Gamma)\le G(K)$, which is impossible by Lemma \ref{lem:droms}.
\end{proof}

\section{Torus knots}\label{sec:torusknot}

Among torus knots, we can restrict ourselves to torus knots of even type by Lemma \ref{lem:kb}.
In this section, we prove the following.

\begin{theorem}\label{thm:torusknot}
Let $K$ be a non-trivial torus knot of even type.
Then a TRAAG $A(\Gamma)$ embeds into $G(K)$ if and only if
$\Gamma$ is a sink star  $S_n$ for some $n\ge 1$.
\end{theorem}

Clearly, a sink star $S_n$ has an induced subgraph $mP_1$ for any $m\ (\le n)$.
Hence, if $A(S_n)$ embeds into $G(K)$, then so does $A(mP_1)$ by Lemma \ref{lem:induce}.
But $mP_1$ has no directed edge, so we exclude this possibility from the conclusion.
(Compare Theorem \ref{thm:katayama}(2).)


First, we confirm that the TRAAG based on a sink star can embed into $G(K)$ for a non-trivial torus knot $K$ of even type.

\begin{proposition}\label{prop:torus1}
Let $n\ge 1$, and let $\Gamma=S_n$ be a sink star with $n+1$ vertices.
If $K$ is a non-trivial torus knot of even type, then
$A(\Gamma)$ can embed into $G(K)$.
\end{proposition}

\begin{proof}
Let $G(K)=\langle x, y \mid x^{2p}=y^q \rangle$ be the standard presentation of $G(K)$.
Note that $x^{2p}$ is central.
Set $z_i=[x^p, (yxy)^i]$ for $i=1,2,\dots$.
Then $z_ix^pz_i=x^p$.
For, 
\begin{align*}
z_i x^p z_i&=[x^p,(yxy)^i]x^p [x^p,(yxy)^i]\\
&=x^{-p} (yxy)^{-i} x^p (yxy)^i \cdot x^p \cdot x^{-p} (yxy)^{-i} x^p (yxy)^i \\
&= x^{-p} (yxy)^{-i} x^{2p}(yxy)^i =x^p.
\end{align*}

We show that $Z=\{z_i \mid i=1,2,\dots, n\}$ generates a free group of rank $n$ in $G(K)$.
First, $Z\cap Z^{-1}=\varnothing$.
This follows from the observation 
\begin{align*}
z_iz_j&=x^{-p}(yxy)^{-i}x^p(yxy)^i\cdot x^{-p}(yxy)^{-j} x^p (yxy)^j\\
&= x^{-p}(y^{-1} (x^{-1}y^{-2})^{i-1} x^{-1}y^{-1}) x^p (y(xy^2)^{i-1}xy) \cdot \\
& \qquad x^{-p}(y^{-1} (x^{-1}y^{-2})^{j-1} x^{-1}y^{-1}) x^p (y(xy^2)^{j-1}xy) \ne 1
\end{align*}
 in the amalgamated free product $G(K)=\langle x \rangle *_{x^{2p}=y^q} \langle y \rangle$.
 (Note that $q>2$.)
Second,  let $w=u_1u_2\cdots u_m$ such that each $u_j\in Z^{\pm 1}$ and $u_j u_{j+1}\ne 1$ for $1\le j <m$.
Then we see $w \ne 1$. 
For, if both $u_j, u_{j+1}\in Z$ or $Z^{-1}$, then  there is no cancellation in the product $u_ju_{j+1}$ as shown above.
Also, if $i\ne j$, then  each of 
\begin{align*}
z_i z_j^{-1}&=x^{-p} (yxy)^{-i} x^p(yxy)^i\cdot (yxy)^{-j}x^{-p} (yxy)^j x^p\\
&=x^{-p}(yxy)^{-i}x^p (yxy)^{i-j} x^{-p}(yxy)^j x^p,   \\
z_i^{-1}z_j&=(yxy)^{-i}x^{-p} (yxy)^i x^p\cdot x^{-p}(yxy)^{-j} x^p (yxy)^j \\
&= (yxy)^{-i}x^{-p}(yxy)^{i-j} x^p (yxy)^j
\end{align*}
yields a non-trivial reduced word.
Thus $w$ is an alternate sequence of $x^{\pm p}$ and $(yxy)^{k} \ (k\ne 0)$, which implies $w\ne 1$ in $G(K)$.

Hence the set $Z$ generates a free subgroup of rank $n$ in $G(K)$ (see \cite[Proposition 1.9]{LS}).
Thus $\langle x^p,Z\rangle\cong A(S_n)$.
\end{proof}

\begin{remark}\label{rem:cable}
A similar argument shows that if $\Gamma=S_n$, then $A(\Gamma)$ can embed into the fundamental group of a cable space of even type.

Let $M$ be a cable space of type $(2p,q)$.
Then $\pi_1(M)=\langle x, y, t \mid x^{2p}=t^{2p}y^q, [y,t]=1 \rangle$, where
$x$ represents the exceptional fiber of index $2p$.
Set $z_i=[x^p, (yxy)^i]$ for $i=1,2,\dots$ as before.
Then $x^{2p}$ is central in $\pi_1(M)$, and $z_ix^pz_i=x^p$.

Note that $\pi_1(M)=\langle x\rangle *_{x^{2p}=t^{2p}y^q} \langle y, t \mid [y,t]=1 \rangle$ is the amalgamated free product,
and $z_i$ is reduced there.  The remaining argument is the same as in the proof of Proposition \ref{prop:torus1}.
\end{remark}

Conversely, we suppose that a TRAAG $A(\Gamma)$ embeds into a torus knot group $G(K)$.
We will show that $\Gamma$ is a sink star.

\begin{lemma}\label{lem:star}
$\bar{\Gamma}$ is a star.
\end{lemma}

\begin{proof}
By Lemma \ref{lem:pride}, the RAAG $A(\bar{\Gamma})$ can embed into $G(K)$.
Then Theorem \ref{thm:katayama} (\cite{K1})  shows that
$\bar{\Gamma}$ is a star.
(Since we are assuming that $\Gamma$ has at least one directed edge, we can exclude the case where
$\bar{\Gamma}$ is an empty graph $mP_1$.)
\end{proof}

If $\Gamma$ has order two, then $\Gamma$ is a sink star.
Hence we assume that $\Gamma$ has at least three vertices.
Let $a$ be the central vertex, and let  $b_1,b_2,\dots, b_n\ (n\ge 2)$ be the end vertices.




\begin{lemma}\label{lem:notail}
There is no directed edge with tail $a$.
\end{lemma}

\begin{proof}
Assume that there is a directed edge $[a,b_i\rangle$.
Then $b_i^x=e^p$ for some $x\in G(K)$, where $e$ is the exceptional fiber of index $2p$
by Lemma \ref{lem:kb}.
Hence $(b_i^2)^x=(b_i^x)^2=e^{2p}=h$, where $h$ is a regular fiber.
Then $b_i^2=h$, because $h$ is central in $G(K)$.
Thus $b_i^2$ is also central.
However, this contradicts the fact that $\langle b_i,b_j\rangle$ with $i\ne j$  is a free subgroup of rank two in $G(K)$ by Lemma \ref{lem:induce}.
\end{proof}

\begin{proposition}\label{prop:torus2}
$\Gamma$ is a sink star.
\end{proposition}

\begin{proof}
By our assumption that $\Gamma$ has at least one directed edge and Lemma \ref{lem:notail}, 
there is a directed  edge $[b_i,a\rangle$.

Suppose that there is an undirected edge $[a,b_j]$ with $i\ne j$ for a contradiction.
Again, $a^x=e^p$ for some $x\in G(K)$, where $e$ is the exceptional fiber of index $2p$,
and $a^2=h$, a regular fiber, as in the proof of Lemma \ref{lem:notail}.

Let $G(K)=\langle e, f \mid e^{2p}=f^q \rangle$.
Then $h=e^{2p}=f^q$.
Consider the natural projection
\[
\phi\colon G(K)\to G(K)/\langle\! \langle h \rangle\!\rangle=\langle e \mid e^{2p}=1 \rangle * \langle f \mid f^q=1 \rangle.
\]
Note that $a$ and $b_j$ commute in $G(K)$, so do $\phi(a)$ and $\phi(b_j)$.
Here, $\phi(a)=\phi((e^p)^{x^{-1}})$.
Thus 
$\phi(b_j)$ lies in a conjugate subgroup $\phi(x) \langle e \mid e^{2p}=1 \rangle \phi(x^{-1})$, or $\phi(a)$ and $\phi(b_j)$ 
are both powers of the same element (see \cite[Corollary 4.1.6]{MKS}).

In the former, $b_j=h^k  (xe^\ell x^{-1})$ for some integers $k,\ell$.
Then $b_j^{2p}= h^{2pk+\ell}$, so $b_j^{2p}$ commutes with $b_i$.
This contradicts that $b_i$ and $b_j$ generate a free group of rank two.

In the latter, set $\phi(a)=u^k$ and $\phi(b_j)=u^\ell$ for some integers $k\ (\ne 0), \ell$.
Then $\phi(a^\ell)=\phi(b_j^k)$, which implies that $a^\ell =h^t b_j^k$ for some integer $t$.
Further, $h^\ell=a^{2\ell}=h^{2t}b_j^{2k}$, so $b_j^{2k}$ is a power of $h$.
Thus $b_j^{2k}$ commutes with $b_i$, which is impossible as above.
\end{proof}

\begin{proof}[Proof of Theorem \ref{thm:torusknot}]
This follows from Propositions \ref{prop:torus1} and \ref{prop:torus2}.
\end{proof}

A similar argument gives the next.

\begin{theorem}\label{thm:cable}
Let $M$ be a cable space of even type, and let $G=\pi_1(M)$.
Then a TRAAG $A(\Gamma)$ embeds into $G$ if and only if
$\Gamma$ is a sink star  $S_n$ for some $n\ge 1$.
\end{theorem}

\begin{proof}
The ``if'' part is given in Remark \ref{rem:cable}.
Suppose that $A(\Gamma)$ embeds into $G$.
Lemmas \ref{lem:star} and \ref{lem:notail} hold again.

As in Remark \ref{rem:cable},
let $G=\langle x, y, t \mid x^{2p}=t^{2p}y^q, [y,t]=1 \rangle$, where
$x$ represents the exceptional fiber of index $2p$.
Set $h=x^{2p}$, which is a regular fiber.
By using the natural projection
\[
\phi\colon G \to G/\langle\!\langle h \rangle\!\rangle=\langle x \mid x^{2p}=1\rangle * \langle y, t \mid [y,t]=1, t^{2p}y^q=1\rangle= \langle x \mid x^{2p}=1\rangle* \mathbb{Z},
\]
the remaining argument goes  on as in the proof of Proposition \ref{prop:torus2}.
\end{proof}

\section{Group theoretic results}

In this section, we prove some technical results for the remaining sections.

\subsection{Amalgamated free product}

For a knot exterior $E(K)$, choose an essential torus $T$ from the torus decomposition.
This torus decomposes $E(K)$ into $M$ and $N$.
(They are not necessarily pieces arisen from the torus decomposition.)
Hence $G(K)=\pi_1(M)*_{\pi_1(T)} \pi_1(N)$.

Keeping this in mind, consider   an amalgamated free product $G=A*_H B$.
Select the right coset representative systems $\{c_i\}$ for $A$ mod $H$ and $B$ mod $H$ (see \cite{MKS}).

\begin{lemma}\label{lem:g1}
For $a\in A-H$, if $a$ is conjugate into $H$, then 
$a^w \in H$ for some $w\in A-H$.
\end{lemma}

\begin{proof}
By the assumption, $a^w\in H$ for some $w\in G$.
If $w\in A-H$, then we are done.
If $w\in H$, then $a\in H$, a contradiction.
Hence we assume $w\not \in H$.

Set $h=a^w$, and let $w=kc_1c_2\dots c_s$ be the reduced form of $w$, where
$k\in H$ (possibly, $k=1$) and $c_i$ are representatives.
In particular, $s\ge 1$, $c_i\in A-H$ or $B-H$, and $c_i$ and $c_{i+1}$ do not belong to the same factor.

Then the proof of \cite[Theorem 4.6(i)]{MKS} claims that
in the sequence
\[
h, h^{c_s^{-1}}, h^{c_s^{-1}c_{s-1}^{-1}},\dots, h^{c_s^{-1}c_{s-1}^{-1}\dots c_3^{-1}c_2^{-1}},
h^{c_s^{-1}c_{s-1}^{-1}\dots c_2^{-1}c_1^{-1}k^{-1}}=a,
\]
any element except the last belongs to $H$.

Let $h_2=h^{c_s^{-1}c_{s-1}^{-1}\dots c_2^{-1}}\in H$.  
(When $s=1$, set $h_2=h$.)
Then $h_2=a^{kc_1}$.
If $c_1\in B-H$, then
\[
h_2=c_1^{-1}k^{-1}\cdot a\cdot kc_1
\]
has representative length $>1$, because $a\in A-H$ and $c_1\in B-H$.
Thus $c_1\in A-H$.
Since $kc_1\in A-H$, we are done.
\end{proof}

\begin{lemma}\label{lem:g2}
Let $a\in A-H$.  Suppose that $a$ is not conjugate into $H$.
If $a^w \in A-H$ for $w\in G$, then $w\in A$.
\end{lemma}

\begin{proof}
Let $w=kc_1c_2\dots c_s$ be the reduced form of $w$.
Suppose $w\not \in A$ for a contradiction.
Then either $c_1\in B-H$, or $c_1\in A-H$ and $c_2\in B-H$.

Suppose $c_1\in B-H$.  Then
\begin{equation}\label{eq:1}
a^w=c_s^{-1}\dots c_2^{-1} c_1^{-1} k^{-1}\cdot a  \cdot kc_1c_2\dots c_s.
\end{equation}
Since $k^{-1}ak \in A-H$, both sides have distinct representative lengths, a contradiction.

Suppose that $c_1\in A-H$ and $c_2\in B-H$.
Then we have $c_1^{-1}k^{-1}akc_1\in A-H$ in the right hand side of (\ref{eq:1}).
Hence both side have distinct representative lengths again.
\end{proof}

\subsection{Generalized torsion pair}

We consider the situation where a knot group $G(K)$ contains
a generalized torsion pair $(g,c)$.
Equivalently, we have a subgroup $\langle g, c \mid gcg=c\rangle$ (\cite{HMT1}).
By Lemma \ref{lem:kb}, the knot exterior $E(K)$ contains
a Seifert fibered piece  $M$ of even type, which is either
a torus knot exterior  or a cable space.
We exclude the case where $M=E(K)$, because the torus knots are done in Section \ref{sec:torusknot}.
After a conjugation,  we may assume that $g, c\in \pi_1(M)$ \cite{HMT2}.
That is, the pair $(g,c)$ is a generalized torsion pair in $\pi_1(M)$.

In this subsection, we prove that neither $g$ nor $c$ lies in $\pi_1(T)$ for any boundary torus 
component $T$ of $M$.

\begin{lemma}\label{lem:gtorsion}
Let $T$ be a boundary component of a Seifert fibered piece $M$ of the knot exterior $E(K)$.
If $(g,c)$ is a generalized torsion pair in $\pi_1(M)$, 
then $g\not \in \pi_1(T)$.
\end{lemma}

\begin{proof}
Suppose $g\in \pi_1(T)$ for a contradiction.
We divide the argument into two cases.

\medskip
\textbf{Case 1.} $M$ is a torus knot exterior of type $(2p,q)$.

On $T=\partial M$, we have a standard meridian-longitude pair $(m,\lambda)$, where
$\lambda$ is null-homologous in $M$.
Let $g=m^i \lambda^j \in \pi_1(T)$.
Then  $g^{c}=g^{-1}$ implies $[g]=0 \in H_1(M)=\mathbb{Z}$.
Since $[g]=i[m]$, we have $i=0$.

We may assume $j> 0$.  Here, we make use of stable commutator length.
For $g=\lambda^j$,  $\mathrm{scl}(\lambda^j)=j\cdot \mathrm{scl}(\lambda) \ge j/2$ (see \cite{C, IMT1, IMT2}).
This contradicts that any generalized torsion element of order two has scl $0$ (\cite{IMT1}).

\medskip
\textbf{Case 2.} $M$ is a cable space of type $(2p,q)$.

For a solid torus $J=S^1\times D^2$,  take a concentric smaller solid torus
$J'=S^1\times D_0$ with $D_0\subset D^2$, and a $(2p,q)$ curve $k$ on $\partial J'$, running $2p$ times along $S^1$.
Then $M$ is homeomorphic to the exterior of $k$ in $J$.
Let $m=\{*\}\times \partial D^2$ and $\ell=S^1\times \{*\}$ be a meridian-longitude pair on $\partial J$,
and let $e$ be the core of $J$ (and $J'$).
Then $\pi_1(M)=\langle m,\ell, e \mid  e^{2p}=m^q \ell^{2p}, [m,\ell]=1 \rangle$.

First, set $T=\partial J$.  Then $\pi_1(T)=\langle m,\ell \mid [m,\ell]=1 \rangle$,
and we can put $g=m^i\ell^j$ for some integers $i$ and $j$.
Since $g^{c}=g^{-1}$, $[g]=0 \in H_1(M)=\langle [\mu], [\ell]\rangle=\mathbb{Z}^2$,
where $\mu$ is a meridian of the curve $k$.
However, $[g]=i[m]+j[\ell]=2pi[\mu]+j[\ell]=0$ implies $i=j=0$, so $g=1$, a contradiction.

Second, let $T$ be the boundary of the tubular neighborhood of $k$.
Let $h$ be a regular fiber of $M$.
Then $\pi_1(T)=\langle \mu, h \mid [\mu,h]=1\rangle$, and
set $g=\mu^i h^j$.
As before, $g^c=g^{-1}$ implies  $[g]=0 \in H_1(M)$.
However,
\begin{align*}
[g]&=i[\mu]+j[h]=i[\mu]+j (2p[\ell]+q[m])\\
&=i[\mu]+j(2p[\ell]+2pq[\mu])=(i+2pqj)[\mu]+2pj[\ell].
\end{align*}
Hence $i=j=0$, so $g=1$ again.
\end{proof}

\begin{remark}
If $(g,c)$ is a generalized torsion pair, then so is $(g^n,c)$ for any $n\ne 0$.
For, $g^c=g^{-1}$ implies $(g^n)^c=(g^c)^n=(g^{-1})^n=(g^n)^{-1}$.
(Of course, we need $g^n\ne 1$, which follows from the fact that $G(K)$ is torsion-free.)
Thus Lemma \ref{lem:gtorsion} shows that $g^n \not \in \pi_1(T)$ for any $n\ne 0$.
\end{remark}

\begin{lemma}\label{lem:c}
Let $T$ be a boundary component of a Seifert fibered piece $M$ of the knot exterior $E(K)$.
If $(g,c)$ is a generalized torsion pair in $\pi_1(M)$, 
then $c\not \in \pi_1(T)$.
\end{lemma}

\begin{proof}
By Lemma \ref{lem:kb},  $c^x=e^p$ for some $x\in \pi_1(M)$ and the (unique) exceptional fiber $e$ with index $2p$.
In particular, $(c^x)^2=e^{2p}$ gives a regular fiber $h$.

\medskip
\textbf{Case 1.} $M$ is a torus knot exterior of type $(2p,q)$.

On $T=\partial M$, choose a standard meridian-longitude pair $(m,\lambda)$, where
$\lambda$ is null-homologous in $M$.
Then $h=m^{2pq}\lambda$.
Since $[e]=q\in H_1(M)=\mathbb{Z}$, we have $[c]=[c^x]=[e^p]=pq$.

Suppose $c\in \pi_1(T)$.  Then $c=m^i\lambda^j$ for some integers $i$ and $j$, and $[c]=i\in H_1(M)$.
Hence $i=pq$.
For $c^2=m^{2pq}\lambda^{2j}$,
consider $c^2h^{-1}=\lambda^{2j-1}\in \pi_1(T)$.
In $\pi_1(M)$, $(c^2h^{-1})^x=(c^x)^2 h^{-1}=1$, so $c^2h^{-1}=1$.
Then $\lambda^{2j-1}=1$, which contradicts that $\pi_1(M)$ is torsion-free.

\medskip
\textbf{Case 2.} $M$ is a cable space of type $(2p,q)$.

As in the proof of Lemma \ref{lem:gtorsion},
we have a presentation $\pi_1(M)=\langle m,\ell, e \mid  e^{2p}=m^q \ell^{2p}, [m,\ell]=1 \rangle$, where
$h=m^{q}\ell^{2p}$.

First, set $T=\partial J$.
Assume $c\in \pi_1(T)$, and set $c=m^i \ell^j$.
Then $[c]=i[m]+j[\ell]=2pi[\mu]+j[\ell]\in H_1(M)=\langle [\mu],[\ell] \rangle$.
Since $c^2=m^{2i}\ell^{2j}$, we have $c^2h^{-1}=m^{2i-q}\ell^{2j-2p}$.
Hence
\[[c^2h^{-1}]=(2i-q)[m]+(2j-2p)[\ell]=2p(2i-q)[\mu]+(2j-2p)[\ell] \in H_1(M).
\]
As before, 
$(c^2h^{-1})^x=(c^x)^2 h^{-1}=1$ in $\pi_1(M)$, so $c^2h^{-1}=1$.
Then $q=2i$ and $p=j$.
However, this is impossible, because $q$ is odd.

Second, let $T$ be the boundary of the neighborhood of the curve $k$.
Then $\pi_1(T)=\langle \mu, h \mid [\mu, h]=1\rangle$.
Assume $c\in \pi_1(T)$.  Then $c=\mu^i h^{j}$, so $c^2=\mu^{2i}h^{2j}$.
Again, consider $c^2h^{-1}=\mu^{2i}h^{2j-1}$.
Then
\begin{align*}
[c^2h^{-1}]&=2i[\mu]+(2j-1)[h]=2i[\mu]+(2j-1)(2p[\ell]+q[m])\\
&=(2i+2p(2j-1)q)[\mu]+2p(2j-1)[\ell]\ne 0\in H_1(M),
\end{align*}
because $2p(2j-1)\ne 0$.
This contradicts $c^2h^{-1}=1$.
\end{proof}







\section{Directed edges in a mixed graph}

By Corollary \ref{cor:h} and Theorem \ref{thm:torusknot},
the remaining case is when $E(K)$ admits a Seifert fibered piece $M\ne E(K)$ of even type.
We suppose this situation throughout this section.

Let $\Gamma$ be a mixed graph with directed edge $[a,b\rangle$.
Suppose that $A(\Gamma)$ embeds into a knot group $G(K)$.
By Lemma \ref{lem:kb}, we may assume that $\langle a, b\rangle \le \pi_1(M)$.
Thus $(a,b)$ gives a generalized torsion pair in $\pi_1(M)$, that is, $a^b=a^{-1}$, equivalently $aba=b$.

There are two possibilities for $M$.
If $M$ is a torus knot exterior, then $T=\partial M$ splits $E(K)$ into $M$ and $N$, say.
That is, $E(K)=M\cup_T N$, so $G(K)=\pi_1(M)*_{\pi_1(T)} \pi_1(N)$.
If $M$ is cable space, then there are two cases.
If $M$ contains $\partial E(K)$, then choose $T$ as another boundary component.
Otherwise, we can choose either boundary component as $T$.
Thus $T$ splits $E(K)$ into $M'$, which contains $M$ (possibly, $M'=M$),  and $N$.
Then $G(K)=\pi_1(M')*_{\pi_1(T)} \pi_1(N)$.
Hence the torus knot exterior case can also be regarded as the latter when $M'=M$.
Since $\pi_1(M)\le \pi_1(M')$, the pair $(a,b)$ is also a generalized torsion pair in $\pi_1(M')$.

\begin{lemma}\label{lem:anott}
Neither $a$ nor $b$ is conjugate into $\pi_1(T)$ in $G(K)$.
\end{lemma}

\begin{proof}
Consider the element $a$.
By Lemma \ref{lem:gtorsion}, $a\not \in \pi_1(T)$.
Assume that $a^x\in \pi_1(T)$ for some $x\in G(K)$.
By Lemma \ref{lem:g1}, we can assume that $x\in \pi_1(M')$.

Suppose that $M$ is a torus knot exterior.  Then $M'=M$.
By taking  a conjugation with $x$, $aba=b$ yields $a^x b^x a^x=b^x$.
Thus a new pair $(a^x, b^x)$ is another generalized torsion pair in $\pi_1(M)$.
This contradicts Lemma \ref{lem:gtorsion}.

Suppose that $M$ is a cable space.
If $M'=M$, then we are done as above.
Otherwise, let $T'$ be the other boundary component of $M$.
Then $M'=M\cup_{T'} L$, so $\pi_1(M')=\pi_1(M)*_{\pi_1(T')} \pi_1(L)$.
We remark that $a\not \in \pi_1(T')$ by Lemma \ref{lem:gtorsion}.

If $a$ is not conjugate into $\pi_1(T')$ in $\pi_1(M')$, then Lemma \ref{lem:g2} implies $x\in \pi_1(M)$.
Again, $(a^x,b^x)$ gives a generalized torsion pair in $\pi_1(M)$, contradicting Lemma \ref{lem:gtorsion}.

Thus $a^y \in \pi_1(T')$ for some $y\in \pi_1(M')$.
However, we can assume that $y\in \pi_1(M)$ by Lemma \ref{lem:g1}.
Then $(a^y,b^y)$ gives a generalized torsion pair in $\pi_1(M)$,
contradicting Lemma \ref{lem:gtorsion} again.

The argument for the element $b$ is the same.  Use Lemma \ref{lem:c} instead of Lemma \ref{lem:gtorsion}.
\end{proof}

\begin{proposition}\label{prop:tail}
Let $\Gamma$ be a mixed graph with directed edge $[a,b\rangle$.
Assume that $A(\Gamma)$ embeds into a knot group $G(K)$.
Then 
\begin{itemize}
\item[(1)]
there is neither an undirected edge $[a,c]$  nor a directed edge $[a,c\rangle$ with $c\ne b$; and
\item[(2)]
there is neither an undirected edge $[b,c]$  nor a directed edge $[b,c\rangle$.
\end{itemize}
\end{proposition}

\begin{proof}
By Lemma \ref{lem:forest}, each connected component of $\Gamma$ is a tree.

(1)
If there is an edge $[a,c]$ or $[a,c\rangle$ with $c\ne b$, then
we have a relation $a^c=a^{\pm 1}$ in $G(K)$.

Consider $G(K)=\pi_1(M')*_{\pi_1(T)} \pi_1(N)$ as above.
We have $a\in \pi_1(M')-\pi_1(T)$ (Lemma \ref{lem:gtorsion}) and $a$ is not conjugate into $\pi_1(T)$ in $G(K)$
(Lemma \ref{lem:anott}).
Lemma \ref{lem:g2} implies $c\in \pi_1(M')$.

If $M'=M$, then we have $a,b,c\in \pi_1(M)$.
This is impossible by Proposition \ref{prop:torus2} and Theorem \ref{thm:cable}.

If $M'\ne M$, then consider a splitting $M'=M\cup_{T'} L$ as in the proof of Lemma \ref{lem:anott}.
Then $a\in \pi_1(M)-\pi_1(T')$.
Since $a$ is not conjugate into $\pi_1(T')$ in $\pi_1(M')$ by Lemma \ref{lem:anott}
(we can choose any boundary component of $M$ as $T$),
 Lemma \ref{lem:g2} implies $c\in \pi_1(M)$, so $a,b,c \in \pi_1(M)$, impossible again.

 (2)
 Since $\Gamma$ is simple (no multiple edges),  $c\ne a$ automatically.
The rest of the argument is the same as (1).  Use Lemma \ref{lem:c} instead of Lemma \ref{lem:gtorsion}.
 \end{proof}

\begin{corollary}\label{cor:direct}
Let $\Gamma$ be a mixed graph.
If $A(\Gamma)$ embeds into a knot group, then each connected component of $\Gamma$ having
a directed edge is a sink star.
\end{corollary}

\begin{proof}
This immediately follows from Proposition \ref{prop:tail}.
\end{proof}

\section{No Seifert-Seifert gluing}

In this section, we consider the case where $E(K)$ contains both a Seifert piece and a hyperbolic piece, and there is no Seifert-Seifert gluing.

\begin{theorem}\label{thm:SH}
Suppose that $E(K)$ contains both a Seifert piece and a hyperbolic piece, and that there is no Seifert-Seifert gluing.
Then a TRAAG $A(\Gamma)$ embeds into $G(K)$ if and only if
there is at least one Seifert fibered piece of even type, and
$\Gamma= mP_1+\sum T_{m_i}+\sum S_{n_j}$ for $m\ge 0$ and $m_i, n_j\ge 1$.  (By our assumption, $\Gamma$ has at least one sink star.)
\end{theorem}

First, we prove the ``if'' part.

\begin{proposition}\label{prop:SHif}
Suppose that $E(K)$ contains both a Seifert piece and a hyperbolic piece, and that there is no Seifert-Seifert gluing.
Let $\Gamma=mP_1+\sum T_{m_i}+\sum S_{n_j}$ for $m\ge 0$ and $m_i,n_j\ge 1$, where there is at least one sink star.
If $E(K)$ contains at least one Seifert fibered  piece of even type, then $A(\Gamma)$ embeds into $G(K)$.
\end{proposition}

\begin{proof}
Let $M$ be a Seifert fibered piece of even type, and $H$ a hyperbolic piece which
shares a torus boundary $T$.
Then $\pi_1(M\cup H)=\pi_1(M)*_{\pi_1(T)} \pi_1(H)$.
We know that $A(S_n)\le \pi_1(M)$ for any $n\ge 1$ by Proposition \ref{prop:torus1} and
Remark \ref{rem:cable}.

Take $a\in \pi_1(M)-\pi_1(T)$ and $b\in \pi_1(H)-\pi_1(T)$ such that
$b^i \not \in \pi_1(T)$ for any non-zero integer $i$.
(The choice of $b$ is attained by taking a loxodromic element  in $\pi_1(H)$.) 
We remark that $\pi_1(T)$ is malnormal in $\pi_1(H)$ \cite{HW}.

\begin{claim}\label{cl:SHif1}
For $x\in A(S_n)$ with $x\ne 1$, $x^{ba}$ gives a reduced word whose first and last lie in $\pi_1(M)-\pi_1(T)$.
\end{claim}

\begin{proof}[Proof of Claim \ref{cl:SHif1}]
If $x\in \pi_1(M)-\pi_1(T)$, then $a^{-1}b^{-1}xba$ is reduced.
Otherwise $x\in \pi_1(T)$.
Then $x^b\in \pi_1(H)-\pi_1(T)$ by the malnormality.
Hence $a^{-1}x^ba$ is reduced.
%
\end{proof}

\begin{claim}\label{cl:SHif}
In $\pi_1(M\cup H)$,
$\langle A(S_n)^{ba}, b\rangle =A(S_n)^{ba} * \langle b \rangle$.
\end{claim}

\begin{proof}[Proof of Claim \ref{cl:SHif}]
Let $w\in \langle A(S_n)^{ba}, b \rangle$ be a non-empty reduced word.
Assume $w=1$.
We can assume that 
\[
w=u_1v_1u_2v_2 \dots u_sv_s,
\]
where $s\ge 1$, $u_i\in A(S_n)^{ba}$ and $v_i \in \langle b\rangle$ after conjugation.
Then we have $u_i=(u_i')^{ba} \ (u_i'\in A(S_n))$ and 
$v_i$ is a non-zero power of $b$.
By Claim \ref{cl:SHif1} and $v_i\in \pi_1(H)-\pi_1(T)$,  this gives a non-empty reduced word, a contradiction.
\end{proof}

Thus we have a free product  subgroup $A(S_n)* \mathbb{Z} \le G(K)$.

For a given mixed graph $\Gamma=mP_1+\sum T_{m_i} +\sum S_{n_j}$, let $n=\max\{m_i,n_j\}$.
For any finite sum $\Gamma'$ of $S_n$, 
$A(\Gamma')=A(S_n)*\dots *A(S_n)$, so
$A(\Gamma')\le A(S_n)*\mathbb{Z}$.
(This can be seen by taking a covering space of a wedge of a $2$-complex $\Delta$ with $\pi_1(\Delta)=A(S_n)$ and a circle.)

If we take sufficiently many copies of $S_n$, then $A(\Gamma)\le A(\Gamma')$ by Lemmas \ref{lem:induce} and \ref{lem:pride}.
Thus we have $A(\Gamma)\le G(K)$.
\end{proof}


\begin{proof}[Proof of Theorem \ref{thm:SH}]
The ``if'' part is shown in Proposition \ref{prop:SHif}.

Conversely, suppose that $A(\Gamma)$ embeds into $G(K)$.
By Lemma \ref{lem:pride} and Theorem \ref{thm:katayama},
we know that 
$\bar{\Gamma}=mP_1+\sum T_{m_i}$.
By Corollary \ref{cor:direct}, each connected component of $\Gamma$ having a directed edge is a sink star.
Hence the conclusion follows.
\end{proof}

\section{Seifert-Seifert gluing}\label{sec:SS}

Finally, we prove the next to complete the proof of Theorem \ref{thm:main}.

\begin{theorem}\label{thm:ss}
Assume that $E(K)$ contains a Seifert-Seifert gluing.
Then a TRAAG $A(\Gamma)$ embeds into $G(K)$ if and only if
there is at least one Seifert fibered piece of even type,
and $\Gamma=\sum S_{n_i}+F$ for $n_i\ge 1$, where $F$ is a forest and
$\Gamma$ has at least one sink star.  (Possibly, $F$ is empty.)
\end{theorem}

%



\subsection{Katayama's embedding of $A(P_4)$}

We assume that  $E(K)$ contains a Seifert-Seifert gluing in this section.
Let $M_1$ and $M_2$ be such Seifert  fibered pieces with a common torus boundary $T$.
Then Katayama \cite{K1} (also \cite{NW}) observes $A(P_4)\le \pi_1(M_1\cup M_2)$.
Since $A(F)\le A(P_4)$  for any forest $F$ \cite{KK},
it implies $A(F)\le \pi_1(M_1\cup M_2)\le G(K)$.

Let $G=\pi_1(M_1\cup M_2)=\pi_1(M_1)*_{\pi_1(T)}\pi_1(M_2)$.
We briefly recall the construction $A(P_4)\le G$ of \cite{K1}.

Let $M$ be $M_1$ or $M_2$.
When $M$ is a torus knot space, 
$\pi_1(M)$ has a subgroup
$A(T_{2g})=Z(\pi_1(M))\times [\pi_1(M),\pi_1(M)]$,
where the commutator subgroup is a free group of rank $2g$.
Note that this subgroup has finite index.
If $M$ is a composing space $D_m\times S^1$, where $D_m$ is the $m$-holed disk,
then $\pi_1(M)=A(T_m)$. 
Finally, if $M$ is a cable space, then it has a finite cyclic cover homeomorphic to
a composing space.
In any case, $\pi_1(M)$ has a finite index subgroup $A(T_{m})=\mathbb{Z}\times F_{m}$ for some $m\ge 2$, where the $\mathbb{Z}$ factor is generated by a regular fiber.
For $M_i$, this subgroup is denoted by $A_i$.

Let $h_i$ be a regular fiber of $M_i$.  Then $h_i\in A_i$.
Note that $h_1\ne h_2^{\pm 1}$ on $T$, because $M_1$ and $M_2$ are adjacent distinct pieces.
We do not know whether $h_2\in A_1$ and $h_1\in A_2$.
But, since $A_i$ has finite index in $\pi_1(M_i)$,
a sufficiently high power of $h_2$ (resp.~$h_1$) lies in $A_1$ (resp.~ $A_2$). 
Let $h_i'$ be such a power.
Finally, take $a_i\in A_i$ so that $\langle a_i, h_j'\rangle$ is free.
(Although we do not know whether $h_j' \in F_{m_i}$, we can take $a_i \in F_{m_i}$.)

Thus we have $A_1'=\langle a_1,h_1', h_2'\rangle\le \pi_1(M_1)$ and $A_2'=\langle a_2,h_1', h_2'\rangle\le \pi_1(M_2)$, which are isomorphic to $A(T_2)=\mathbb{Z}\times F_2$, and $A_1'\cap A_2'=\langle h_1', h_2'\rangle=\mathbb{Z}^2$.
Hence $\langle A_1', A_2'\rangle$ gives $A(P_4)$.


Under this situation, we prove a technical result.
We keep using the above notation.

\begin{lemma}\label{lem:peri1}
Let $H=\langle h_1 \rangle \le \pi_1(T)$.
There exists an element $a\in \pi_1(M_1)-\pi_1(T)$ such that
$t^a\not \in \pi_1(T)$ for any $t \in \pi_1(T)-H$ and that $\langle a,a_1,h_2'\rangle$
is a free group of rank three.
\end{lemma}

In fact, $\langle a,a_1, h_1', h_2'\rangle =A(T_3)\le \pi_1(M_1)$.

\begin{proof}
We divide the argument into three cases.

\medskip
\textbf{Case 1.} $M_1$ is a torus knot exterior.

Recall that $A_1=\mathbb{Z}\times F_{m_1}$, where $F_{m_1}$ is the commutator subgroup of $\pi_1(M_1)$.
Since $m_1\ge 2$, $F_n\le F_{m_1}$ for any $n$.
Hence we can choose $a\in F_{m_1}$ so that $\langle a, a_1,h_2'\rangle$ is free.
This implies $a\not \in \pi_1(T)$, otherwise $a$ and $h_2'$ would commute.

We verify that $t^a\not \in \pi_1(T)$ for any $t\in \pi_1(T)-H$.

Let $\mu,\lambda$ be the meridian and longitude pair on $T$.
For a given $t=\mu^ih_1^j$ with $i\ne 0$, suppose 
$t^a=a^{-1}\mu^i h_1^j a =(\mu^i)^ah_1^j \in \pi_1(T)$.
Then $(\mu^i)^ah_1^j = \mu^k h_1^\ell$ for some integers $k,\ell$.
Hence $(\mu^i)^a\in \pi_1(T)$.
Then $(\mu^i)^{a\lambda}=(\mu^i)^a$, so $(\mu^i)^{a\lambda a^{-1}}=\mu^i$.
This means that $a\lambda a^{-1}$ commutes with $\mu^i$.
Since the centralizer of $\mu^i$ is $\pi_1(T)$ (\cite{AFW}),  $a\lambda a^{-1}\in \pi_1(T)$.
By homological reason, $a\lambda a^{-1}$ is a power of $\lambda$.
However, such a relation $a\lambda a^{-1}=\lambda^s$ is impossible in the commutator subgroup which is free,
because $a$ is not a power of $\lambda$.

\medskip
\textbf{Case 2.} $M_1$ is a composing space.

Then $M_1=D_{m_1}\times S^1$, and $A_1=\pi_1(M_1)$.
Hence $h_2'=h_2$.
(Possibly, $h_1'$ is a power of $h_1$.)
In this case, we can choose $a,a_1 \in \pi_1(D_{m_1})$ so that
$\langle a, a_1, h_2\rangle$ is a free group of rank three in $\pi_1(M_1)$.
In particular,  $a\not \in \pi_1(T)$.

Again, 
we see that $t^a\not \in \pi_1(T)$ for $t\in \pi_1(T)-H$.
For $\pi_1(T)$, choose the generators $\{h_1, d\}$, where $d$ is represented by  a component of $\partial D_{m_1}$.
Let $t=d^i h_1^j$ with $i\ne 0$.
Suppose  $t^a=a^{-1} d^i h_1^j a \in \pi_1(T)$.
Then $a^{-1}d^ih_1^j a=d^k h_1^\ell$ for some integers $k$ and $\ell$.
  Since $h_1$ is central, $(d^i)^a=d^kh_1^{\ell-j}$, so
$d^{-k}(d^i)^a=h_1^{\ell-j}$.
Since $a, d\in \pi_1(D_{m_1})$, $d^{-k}(d^i)^a \in \pi_1(D_{m_1})$.
Thus $\ell=j$, because $h_1$ is a regular fiber representing the $S^1$-factor of $A_1$.
Then we have a relation $a^{-1} d^i a=d^k$ in $\pi_1(D_{m_1})$.
This is impossible, because $a$ is not a power of $d$.

\medskip
\textbf{Case 3.}  $M_1$ is a cable space.

The cable space $M_1$ is obtained from a solid torus $D^2\times S^1$ by removing
a tubular neighborhood of a torus knot of type $(p,q)$,  lying on a smaller concentric torus.
A punctured meridian disk $D_p \ (p\ge 2)$ corresponds to the $F_p$-factor of $A_1=F_p\times \mathbb{Z}$.
The $\mathbb{Z}$-factor comes from a regular fiber.
Hence we can choose $a, a_1 \in \pi_1(D_p)$ so that $\langle a,a_1,h_2'\rangle$ is a free group of rank three.

We may assume that $T$ is the inner boundary component by the symmetry of $M_1$.
Again, we know $a\not \in \pi_1(T)$.
Let $\mu, \lambda$ be the meridian and longitude  pair on $T$.
For $t\in \pi_1(T)-H$,  assume that $t^a\in \pi_1(T)$.
If $t=\mu^i h_1^j$ with $i\ne 0$, then $t^a=a^{-1}\mu^i h_1^j a$. 
Thus $a^{-1}\mu^i h_1^j a=\mu^k h_1^\ell$, so $(\mu^i)^a\in \pi_1(T)$.
As in Case 1,  we have $a\lambda a^{-1} \in \pi_1(T)$, which implies $a\lambda a^{-1}=\lambda^s$ for some $s$ from homological reason.

Consider the finite cyclic cover $\widetilde{M_1}$  of $M_1$, corresponding
to the kernel of $\pi_1(M_1) \to H_1(M_1)=\langle \mu, c\rangle  \to \mathbb{Z}_p$,
sending $\mu\mapsto 0$ and $c\mapsto 1$, where
$c=\{*\}\times S^1$ on $\partial M_1-T$.
Then $\widetilde{M_1}\cong D_p \times S^1$, and the relation $a\lambda a^{-1} =\lambda^s$ yields
$\tilde{a}\tilde{\lambda}\tilde{a}^{-1}=\tilde{\lambda}^s$ for suitable lifts $\tilde{a}, \tilde{\lambda}$ of $a$ and $\lambda$.  This implies $s=1$, because $\pi_1(\widetilde{M_1})=F_p\times \mathbb{Z}$.
Hence we have $a\lambda a^{-1}=\lambda$ in $\pi_1(M_1)$.
However, this is impossible, because the centralizer of $\lambda$ is $\pi_1(T)$ and $a\not \in \pi_1(T)$
\end{proof}

\subsection{Embedding of $A(S_n+P_4)$}

From now, we assume that $M_1$ is a Seifert fibered piece of even type, and
$A(S_n) \le \pi_1(M_1)$.
By Lemma \ref{lem:peri1}, we take an element $a \in \pi_1(M_1)-\pi_1(T)$.
Also, applying Lemma \ref{lem:peri1} to $M_2$,  we have an element $b\in \pi_1(M_2)-\pi_1(T)$
with a similar property to $a$.
Recall that $h_i$ is a regular fiber of $M_i$, and set $H_i=\langle h_i \rangle\le \pi_1(T)$.

\begin{lemma}\label{lem:embed1}
For $x \in A(S_n)$ with $x\ne 1$, $x^{bab}$  gives a reduced word in the amalgamated free product $\pi_1(M_1\cup M_2)$
whose first and last lie in $\pi_1(M_2)-\pi_1(T)$.
\end{lemma}

\begin{proof}
If $x\in \pi_1(M_1)-\pi_1(T)$, then $b^{-1}a^{-1}b^{-1}x bab$ satisfies the conclusion.
If $x \in \pi_1(T)-H_2$, then $x^b \in \pi_1(M_2)-\pi_1(T)$.
Thus $b^{-1}a^{-1} x^b ab$ gives the reduced word.
Finally, if $x\in H_2$, then  $x\not \in H_1$ and
$x^{ba}=x^a \in \pi_1(M_1)-\pi_1(T)$, so
$b^{-1}x^{ba}b$ is reduced.
\end{proof}

\begin{lemma}\label{lem:embed2}
If $y\in A(P_4)$ with $y\ne 1$, then $y^{baba}$ gives a reduced word in 
 the amalgamated free product $\pi_1(M_1\cup M_2)$
whose first and last lie in $\pi_1(M_1)-\pi_1(T)$.
\end{lemma}

\begin{proof}
Recall that $A(P_4)=\langle a_1, h_1', h_2', a_2\rangle$.
First, we deal with the case where $y$ does not  contain a power of $a_1$.
Then $y\in \pi_1(M_2)$.
If $y^b \in \pi_1(M_2)-\pi_1(T)$, then 
$a^{-1}b^{-1}ay^b aba$ is reduced.
Suppose $y^b\in \pi_1(T)$.
If $y^b\not \in H_1$, then $(y^b)^a\in \pi_1(M_1)-\pi_1(T)$.
Thus $a^{-1}b^{-1}y^{ba}ba$ is reduced.
If $y^b \in H_1$, then $y^b \not \in H_2$, and $(y^b)^{ab}=(y^b)^b\in \pi_1(M_2)-\pi_1(T)$.
Therefore $a^{-1} y^{bab}a$ is reduced.

Thus we may assume that $y$ contains a non-zero power of $a_1$.
We write 
\begin{equation}\label{eq:y}
y=y_1 a_1^{c_1} y_2 a_1^{c_2}\dots y_{d-1}a_1^{c_{d-1}} y_d,
\end{equation}
where 
$y_i$ contains no $a_1^{\pm 1}$, 
$y_i\ne 1 \ (2\le i\le d-1)$, $c_i\in \mathbb{Z}-\{0\}$, and possibly $y_1=1$ or $y_d=1$.
Furthermore, we may assume that $y_i\ (2\le i\le d-1)$ is not a power of $h_1'$, since
$[a_1,h_1']=1$.
If $y_1$ is a non-zero power of $h_1'$, then $y_1$ is merged into $y_2$, or $y=a_1^{c_1}(h_1')^j$.
For the latter,  since $a_1^{c_1}\not \in T$, $y\in \pi_1(M_1)-\pi_1(T)$, then we are done.
Similarly, if $y_d$ is a power of $h_1'$, then $y_d$ is merged into $y_{d-1}$, or
$y=a_1^{c_{d-1}} (h_1')^j$.
The latter can be treated as above.
Hence, we may assume that any $y_i$ is not a non-zero power of $h_1'$.

Now, 
\begin{equation}\label{eq:y2}
y^{baba}=a^{-1}b^{-1}a^{-1}\cdot (b^{-1} y_1)\cdot a_1^{c_1} y_2 a_1^{c_2}\dots y_{d-1}a_1^{c_{d-1}}\cdot
( y_d b)\cdot aba.
\end{equation}

\begin{claim}\label{cl:embed1}
$b^{-1}y_1, y_d  b \in \pi_1(M_2)-\pi_1(T).$ 
\end{claim}

\begin{proof}[Proof of Claim \ref{cl:embed1}]
If $y_1=1$, then we are done.
Assume $y_1\ne 1$.
Since $y_1$ does not contain $a_1^{\pm 1}$, and $h_2'$ is central in $\pi_1(M_2)$,
we can write $y_1=(h_2')^i \hat{w}(h_1',a_2)$, where
$\hat{w}(h_1',a_2)$ involves only $(h_1')^{\pm 1}$ and $a_2^{\pm 1}$.
Possibly, $\hat{w}(h_1',a_2)=1$, and either letter does not appear.

Then $b^{-1}y_1=b^{-1}(h_2')^i \hat{w}(h_1',a_2)=(h_2')^i b^{-1} \hat{w}(h_1',a_2)$.
Note that $b^{-1}\hat{w}(h_1',a_2)\ne 1$ and that 
$b^{-1}\hat{w}(h_1',a_2)$ is not a power of $h_1'$,
because $\langle h_1',b,a_2\rangle$ is a free group of rank three.

If $b^{-1}y_1$ lies in $\pi_1(T)$, then $b^{-1}\hat{w}(h_1',a_2)\in \pi_1(T)$.
However, $[h_1', b^{-1}\hat{w}(h_1',a_2)]=1$ in $\pi_1(T)$, which
contradicts that $\langle h_1',b,a_2\rangle$ is a free group of rank three.

The argument for $y_db$ is similar.
\end{proof}

Here, we introduce a symbol $w(\ast,\ast,\dots)$.
For example, $w(h_2')$ is a non-zero power of $h_2'$, $w(h_1',h_2')$ is a word consisting
of only a non-zero power of $h_1'$ and $h_2'$, both of which appear.
We can suppose that such a word has the shortest length among the words representing the same element in $\pi_1(M_1\cup M_2)$.

For $y_i$ defined in (\ref{eq:y}),
there are five possibilities:
\begin{itemize}
\item[(1)] $w(h_2')$, $w(h_1',h_2')$,
\item[(2)] $w(a_2)$,
\item[(3)] $w(h_1',a_2)$,
\item[(4)] $w(h_2',a_2)=(h_2')^i a_2^j\ (i,j\ne 0)$,
\item[(5)] $w(h_1',h_2',a_2)=(h_2')^i w(h_1',a_2)\ (i\ne 0)$.
\end{itemize}

The remaining is to show that the middle segment of (\ref{eq:y2}),
\[
a_1^{c_1} y_2 a_1^{c_2}\dots  y_{d-1}a_1^{c_{d-1}},
\]
gives a reduced word whose first and last lie in $\pi_1(M_1)-\pi_1(T)$.

First, if $y_i$ has a form of (2), that is, a power of $a_2$, then $y_i\in \pi_1(M_2)-\pi_1(T)$.
 (4) is similar.
For (3), if $y_i=w(h_1',a_2)\in \pi_1(T)$, then $[h_1', y_i]=1$.
Since $y_i$ is not a power of $h_1'$, this contradicts that $\langle h_1', a_2\rangle$ is a free
group of rank two.
For (5), if $y_i=(h_2')^i w(h_1',a_2)\in \pi_1(T)$, then  
$w(h_1',a_2) \in \pi_1(T)$, so $[h_1', w(h_1',a_2)]=1$, a contradiction again.
Thus whenever $y_i$ has a form of (2)--(5), $y_i\in \pi_1(M_2)-\pi_1(T)$.

Finally, among $y_2,\dots, y_{d-1}$, we look at them of type (1).
If $y_2$ is of type (1) but $y_3$ is not of type (1), then
set $U=a_1^{c_1}y_2a_1^{c_2}$.
Clearly, $U\in \pi_1(M_1)$.
Assume $U\in \pi_1(T)$ for a contradiction.
We can write $U=(h_1')^i w(h_2',a_1)$ with $w(h_2',a_1)\ne 1$.
Then $w(h_2',a_1)\in \pi_1(T)$, so $[h_2', w(h_2',a_1)]=1$, which
contradicts that $\langle h_2',a_1\rangle$ is a free group of rank two.
Thus $U\in \pi_1(M_1)-\pi_1(T)$.
This argument works when $y_i$ is of type (1), but neither $y_{i-1}$ nor $y_{i+1}$ is of type (1).

If there are successive $y_i$'s of type (1), $y_i, y_{i+1},\dots, y_{i+k}$, say, then  
set 
\[
U=a_1^{c_{i-1}} y_i \dots  y_{i+k} a_1^{c_{i+k}}.
\]
(Here, $y_{i-1}$ and $y_{i+k+1}$ are not of type (1).)
Then the same argument shows $U\in \pi_1(M_1)-\pi_1(T)$.
\end{proof}

We prove that $A(S_n + P_4)$ can embed into $\pi_1(M_1\cup M_2)$.

\begin{proposition}\label{prop:ss-final}
In $\pi_1(M_1\cup M_2)$, 
\[
\langle A(S_n)^{bab}, A(P_4)^{baba}\rangle =A(S_n)^{bab} * A(P_4)^{baba}.
\]
\end{proposition}

\begin{proof}
The argument is similar to that of Claim \ref{cl:SHif}.  (Use Lemmas \ref{lem:embed1} and \ref{lem:embed2}.) \end{proof}



\subsection{Final argument}

\begin{theorem}\label{thm:ssif}
Assume that $E(K)$ contains a Seifert-Seifert gluing, and there is at least one Seifert fibered piece
of even type.
Let $\Gamma=\sum S_{n_i} +F$ for $n_i\ge 1$, where $F$ is a forest and $\Gamma$ has at least one sink star.
Then $A(\Gamma)$ can embed into $G(K)$.
\end{theorem}

\begin{proof}
First, assume that $E(K)$ contains a Seifert fibered piece $M_1$ of even type which
is glued to a Seifert fibered piece $M_2$.
Let $n=\max\{n_i\}$.
By Proposition \ref{prop:ss-final}, we can embed $A(S_n+P_4)$ into $\pi_1(M_1\cup M_2)\le G(K)$.
Recall that $A(F+P_1)\le A(P_4)$ (\cite{KK}), since $F+P_1$ is still a forest.
Thus 
$A(S_n)*A(P_1)*A(F)\le A(S_n)*A(P_4)\le G(K)$.
Here, we can embed $A(\sum S_{n_i})\le A(S_n)*\mathbb{Z}$ as in the proof of Proposition \ref{prop:SHif}.
Hence 
\[
A(\Gamma)=A(\sum S_{n_i}+F)=A(\sum S_{n_i})*A(F)\le A(S_n)*A(P_1)*A(F)\le G(K).
\]

Second, we assume that there is no Seifert fibered piece of even type which is adjacent to a Seifert fibered piece.
Let $M$ be a Seifert fibered piece of even type.   We have $A(S_n)\le \pi_1(M)$ for any $n$.
By the assumption, there is a Seifert-Seifert gluing.
Let $M_1$ and $M_2$ be such Seifert fibered pieces.
We know that $A(F)\le A(P_4)\le \pi_1(M_1\cup M_2)$.

Let us choose a boundary component $T$ of $\partial M$ such that
$T$ splits $E(K)$ into $N_1$ and $N_2$, and
$M\subset N_1$, $M_1\cup M_2\subset N_2$.
In particular,  $T$ is a boundary component of a hyperbolic piece $H$.
Then $G(K)=\pi_1(N_1)*_{\pi_1(T)} \pi_1(N_2)$.
We remark that $\pi_1(T)$ is malnormal in $\pi_1(N_2)$ \cite{HW}.

Take $a\in \pi_1(N_1)-\pi_1(T)$ and $b\in \pi_1(N_2)-\pi_1(T)$.

\begin{claim}\label{cl:last1}
\begin{itemize}
\item[(1)]
For $x\in A(S_n)$ with $x\ne 1$,  $x^{ba}$ gives a reduced word whose first and last lie in $\pi_1(N_1)-\pi_1(T)$.
\item[(2)]
For $y\in A(P_4)$ with $y\ne 1$,  $y^{ab}$ gives a reduced word whose first and last lie in $\pi_1(N_2)-\pi_1(T)$.
\end{itemize}
\end{claim}

\begin{proof}[Proof of Claim \ref{cl:last1}]
(1)
The argument is the same as (1) of Claim \ref{cl:SHif1}.

(2) If $y\in \pi_1(N_2)-\pi_1(T)$, then $b^{-1}a^{-1}yab$ gives a reduced word.
Suppose $y\in \pi_1(T)$.
If $y^a \in \pi_1(N_1)-\pi_1(T)$, then $b^{-1}y^ab$ is reduced.
Otherwise, $(y^a)^b \in \pi_1(N_2)-\pi_1(T)$ by the malnomality.
Thus $y^{ab}$ itself is reduced.
\end{proof}

As in Proposition \ref{prop:ss-final}, 
we can show that $\langle A(S_n)^{ba}, A(P_4)^{ab} \rangle = \langle A(S_n)^{ba} \rangle * \langle A(P_4)^{ab} \rangle$ in $G(K)$. 
The rest of the argument is the same as the first situation.
\end{proof}

\begin{proof}[Proof of Theorem \ref{thm:ss}]
The  ``if'' part is Theorem \ref{thm:ssif}.
We prove the ``only if' part.
By Lemma \ref{lem:pride} and Theorem \ref{thm:katayama},
$\bar{\Gamma}$ is a forest.
On the other hand, each connected component of $\Gamma$ having a directed edge is a sink star by Corollary \ref{cor:direct}.
 \end{proof}

\begin{proof}[Proof of Theorem \ref{thm:main}]
This immediately follows from 
Corollary \ref{cor:h}, Theorems \ref{thm:torusknot},
\ref{thm:SH} and \ref{thm:ss}.
\end{proof}

\section*{Acknowledgment}

We would like to thank Ryoya Kai
for helpful conversations.

\bibliographystyle{amsplain}

\end{document}